\documentclass[letterpaper, 10 pt, conference]{ieeeconf}  

\IEEEoverridecommandlockouts                              
\usepackage{bbding}

\title{\LARGE \bf
Comparison theorem for infinite-dimensional linear impulsive systems*
}

\author{Vladyslav Bivziuk$^{1}$, Sergey Dashkovskiy\,\Envelope\,$^{2}$ \; and Vitalii Slynko$^{2}$
\thanks{*This work was partially supported by the DFG grant number SL343/1-1. \Envelope\, Corresponding author. }
\thanks{$^{1}$Vladyslav Bivziuk is with the Department of Mathemetics, University of Illinois Urbana-Champaign, Champaign, IL 61820, USA
        {\tt\small bivziuk2@illinois.edu}}%
\thanks{$^{2}$Sergey Dashkovskiy and Vitalii Slynko are with the Institute of Mathematics, University of Würzburg,
        Germany 
        {\tt\small name.surname@mathematik.uni-wuerzburg.de}}%
}

\usepackage[utf8]{inputenc}
\usepackage{cite}
\usepackage{amsmath, amsfonts, amssymb}
\usepackage{amsthm}
\usepackage[english]{babel}

\newcommand{\Ree}{\operatorname{Re}}

\newcommand{\T}{\operatorname{T\,}}

\newcommand{\id}{\operatorname{id\,}}

\theoremstyle{plain}
\newtheorem{theorem}{Theorem}
\newtheorem{lemma}{Lemma}
\newtheorem{proposition}{Proposition}

\theoremstyle{definition}
\newtheorem{definition}{Definition}

\begin{document}

\maketitle
\thispagestyle{empty}
\pagestyle{empty}

\begin{abstract}
We consider a linear impulsive system in an infinite-dimensional Banach space. It is assumed that the moments of impulsive action satisfy the averaged dwell-time condition and the linear operator on the right side of the differential equation generates an analytic semigroup in the state space. Using commutator identities, we prove a comparison theorem that reduces the problem of asymptotic stability of the original system to the study of a simpler system with constant dwell-times. An illustrative example of a linear impulsive system of parabolic type in which the continuous and discrete dynamics are both unstable is given.

\end{abstract}


\section{Introduction} 
The study of the stability of hybrid systems \cite{TGS2012} plays an important role in control theory. The class of hybrid systems usually includes impulsive systems \cite{SP1995,LBS1989} and systems with switching \cite{L2003}. Lie algebraic methods for stability investigations of finite-dimensional switching systems were previously used in \cite{NaBa1994,MaAg2000,Gur1995,L2003,ABL2012}. For linear impulsive systems with bounded operators on the right side, the commutator calculus methods were used in \cite{SBT2019,SlB2019,SlB2020}. The ISS property and stability of coupled nonlinear impulsive infinite-dimensional input systems were also studied in \cite{DM2013,DS2022,DS2023,MGL2018}.

The aim of this paper is to extend the results of \cite{SBT2019} to a wider class of linear impulsive systems which includes partial differential equations for which the influence of impulsive disturbances is not well understood. The main contribution of this article is the substantiation of the comparison principle which reduces the problem of the stability of a linear impulsive system for which the sequence of moments of impulsive action satisfies the averaged dwell-time (ADT) condition to the study of the stability of a linear impulsive system with constant dwell-time. This problem is much simpler and can be solved on the basis of the Lyapunov function method from the class of piecewise time-differentiable functions.

In order to derive the main result we apply the Hadamard's commutator formula from \cite{Mag54}, which is extended here to the case of analytic semi-groups.

The work consists of six sections. In the second section, we state the problem; in the third section, we prove an auxiliary result which extends the Hadamar's formula for some classes of unbounded operators. In the fourth section, the main result is proved, and in the fifth section, some examples are given. The sixth section contains conclusions.

\section{Problem statement} 
To state the problem we use the following notation. $\mathbb Z_+$ denotes non-negative integers. Let $X$ be a Banach space, by $L(X)$ we denote the set of linear bonded operators from $X$ to $X$. By $B_r(0)$ we denote the ball of radius $r\ge0$ around the origin. The commutator of $D,C\in L(X)$ is defined by $[D,C]:=DC-CD$. For symmetric square matrices $\bold P,\bold Q$ we write $\bold P\prec \bold Q$ if and only if the matrix $\bold Q-\bold P$ is positive definite.

 We consider the linear impulsive system
\begin{equation}\label{DEq-1}
\gathered
\dot x(t)=Ax(t),\quad t\ne\tau_k,\quad x(\tau_0^+)=x_0\in\mathrm{D}(A),\\
x(t^+)=Bx(t),\quad t=\tau_k,
\endgathered
\end{equation}
where $x\in X$ is a state vector, $A$ is a closed densely defined linear operator with domain
$\mathrm{D}(A)$ that generates an analytic semigroup $(T_t)_{t\in\Bbb R_+}\subset L(X)$ in the space $X$.
Assume that $B$ is a closed linear operator from $\mathrm D(B)$ to $X$. Since we are considering classical solutions of \eqref{DEq-1}, we assume that $B\mathrm{D}(A)\subseteq \mathrm{D}(A)$. 
Here, $\{\tau_k\}_{k=0}^{\infty}$ is a sequence of moments of impulsive action which is assumed to be increasing and having a single accumulation point at infinity. For this sequence, we assume that the ADT condition is satisfied in the following form: there are constants $\theta>0$ and $\chi_{\max}\in[0,\theta)$ such that for all $k\in\Bbb Z_+$, the following inequality holds
\begin{equation}\label{ADT}
\gathered
|\tau_k-\tau_0-k\theta|\le\chi_{\max}.
\endgathered
\end{equation}

We recall the well-known fact that a closed densely defined linear operator $A$ is a generator of an analytic semigroup if and only if it is sectorial in the sense of the following definition.
Let $R_{A}(\lambda)$ be the resolvent of the operator $A$ and $\rho(A)$ be the resolvent set of the operator $A$.

\begin{definition}[\cite{Pazy}]
A closed linear operator $A$ with $\overline{\mathrm{D}(A)}=X$
is called sectorial if for some $a\in\Bbb R$ and $\phi\in(\frac{\pi}{2},\pi)$ it holds that
\begin{equation*}
\gathered
\Sigma_{a,\phi}:=\{\lambda\in\Bbb C\setminus\{a\}\,|\,|\arg(\lambda-a)|<\phi\}\subset\rho(A),
\endgathered
\end{equation*}
 and there is a positive constant $K$ such that for all $\lambda\in\Sigma_{a,\phi},$
the following inequality holds
\begin{equation}\label{Sector}
\gathered
\|R_{A}(\lambda)\|_{L(X)}\le\frac{K}{|\lambda-a|}.
\endgathered
\end{equation}
\end{definition}

\section{Commutators} 
It is known that linear bounded operators $D,C\in L(X)$ satisfy the identity (Hadamard's formula from \cite{Mag54})
\begin{equation}\label{HF}
\gathered
De^{tC}=e^{tC}\sum\limits_{m=0}^{\infty}\frac{t^m}{m!}\{D,C^m\}.
\endgathered
\end{equation}
Here, $\{D,C^m\}$, $m\in\Bbb Z_+$ is a sequence of nested commutators defined recurrently
\begin{equation*}
\gathered
\{D,C^0\}:=D,\quad \{D,C^{m+1}\}:=[\{D,C^{m}\},C],\quad m\in\Bbb Z_+.
\endgathered
\end{equation*}
We define an extension of the operator $\{B,A^m\}$, $m\in \Bbb Z_+$ for the case of unbounded operators $A$ and $B$ inductively. Let $\{B,A^0\}:=A$. Let a linear operator $\{B,A^m\}$ with domain $\mathrm{D}(\{B,A^m\})$ be already defined for some $m\in\Bbb Z_+$. We denote
\begin{equation*}
\gathered
\widehat{\mathrm{D}}:=\{x\in X\,\,|\, Ax\in \mathrm{D}(\{B,A^m\}),\quad \{B,A^m\}x\in\mathrm{D}(A)\,\}.
\endgathered
\end{equation*}
Let $\widehat{\mathrm{D}}\supseteq\mathrm{D}(A)$ and a linear operator with domain $\widehat{\mathrm{D}}$ acting by the rule
\begin{equation*}
\widehat{\mathrm{D}}\ni x\mapsto \{B,A^{m}\}Ax-A\{B,A^{m}\}x.
\end{equation*}
be closable and its closure is denoted by $\{B,A^{m+1}\}$.

We note that by definition $\mathrm{D}(\{B,A^{m}\})\supseteq\mathrm{D}(A)$.

\begin{lemma}
Assume that ,
the operators $\{B,A^m\}$ are defined for all $m\in\Bbb Z_+$, satify the condition $\{B,A^m\}\mathrm{D}(A)\subseteq\mathrm{D}(A)$ and
\begin{equation}\label{ComIneq}
\|\{B,A^{m}\}R_{A}(\lambda)\|_{L(X)}\le\frac{K_1\eta^m}{|\lambda-a|}
\end{equation}
for some constants $\eta>0$, $K_1>0$ and all $\lambda\in\Sigma_{a,\phi}$.

 Then for all $x\in \mathrm{D}(A)$,
the following equality holds
\begin{equation}\label{HF-2}
BT_tx=T_t\sum\limits_{m=0}^{\infty}\frac{t^m}{m!}\{B,A^m\}x.
\end{equation}
\end{lemma}

\begin{proof}
Without loss of generality, we can assume that $a=0$.
First of all, using the method of mathematical induction, we show that for any $N\in\mathbb Z_+$ the identity
\begin{equation}\label{HF-3}
\begin{gathered}
    BT_tx=T_t\sum\limits_{m=0}^{N}\frac{t^m}{m!}\{B,A^m\}x\\ +\frac{1}{2\pi i}\int\limits_{\Gamma}e^{\lambda t}R_A^{N+1}(\lambda)\{B,A^{N+1}\}R_A(\lambda)x\,d\lambda.
\end{gathered}
\end{equation}
holds for all $x\in\mathrm{D}(A)$. Here and further we denote,
\begin{equation*}
\int\limits_{\Gamma}F(\lambda)x\,d\lambda:=\lim\limits_{R\to+\infty}
\int\limits_{\Gamma_R}F(\lambda)x\,d\lambda,
\end{equation*}
where $F\,:\,\rho(A)\to L(X)$, $\Gamma=\Gamma(r,\psi)=\Gamma_1\cup\Gamma_2\cup\Gamma_3$,
$\Gamma_R=\Gamma\cap\overline{B_R(0)}$, $\psi\in (\frac{\pi}{2},\phi)$, $r>0$, $R>0$,
\begin{equation*}
\gathered
\Gamma_1=\Gamma_1(r,\psi)=\{\lambda=-se^{i\psi}\,|\,s\in(-\infty,-r]\},\\
\Gamma_2=\Gamma_2(r,\psi)=\{\lambda=re^{i\alpha}\,|\,\alpha\in[-\psi,\psi]\},\\
\Gamma_3=\Gamma_3(r,\psi)=\{\lambda=se^{i\psi}\,|\,s\in(r,\infty)\},\\
\endgathered
\end{equation*}
 Indeed, using the Dunford--Taylor formula \cite{Pazy}, we obtain
\begin{equation}\label{DSh}
T_tx=\frac{1}{2\pi i}\int\limits_{\Gamma}e^{\lambda t}R_A(\lambda)x\,d\lambda.
\end{equation}

Since $T_tx\in\mathrm{D}(A)\subset\mathrm{D}(B)$ and the operator $B$ is closed, we have
\begin{equation*}
BT_tx=\frac{1}{2\pi i}\int\limits_{\Gamma}e^{\lambda t}BR_A(\lambda)x\,d\lambda.
\end{equation*}
Taking into account the assumption $B\mathrm{D}(A)\subseteq\mathrm{D}(A)$, we get for all $x\in\mathrm D(A)$ that
\begin{equation*}
\gathered
BR_A(\lambda)x-R_A(\lambda)Bx=R_A(\lambda)(\lambda\id-A)BR_A(\lambda)x \\
-R_A(\lambda)B(\lambda\id-A)R_A(\lambda)x=R_A(\lambda)((\lambda\id-A)B\\ 
-B(\lambda\id-A))R_A(\lambda)x
=R_A(\lambda)[B,A]R_A(\lambda)x.
\endgathered
\end{equation*} 

Substituting the expression for $BR_A(\lambda)x$ into \eqref{DSh}, we get
\begin{equation*}
\gathered
BT_tx=\frac{1}{2\pi i}\int\limits_{\Gamma}e^{\lambda t}R_A(\lambda)Bx\,d\lambda\\
+\frac{1}{2\pi i}\int\limits_{\Gamma}e^{\lambda t}R_A(\lambda)[B,A]R_A(\lambda)x\,d\lambda\\=
\frac{1}{2\pi i}\int\limits_{\Gamma}e^{\lambda t}R_A(\lambda)\,d\lambda Bx\\
+\frac{1}{2\pi i}\int\limits_{\Gamma}e^{\lambda t}R_A(\lambda)[B,A]R_A(\lambda)x\,d\lambda\\=
T_tBx+\frac{1}{2\pi i}\int\limits_{\Gamma}e^{\lambda t}R_A(\lambda)[B,A]R_A(\lambda)x\,d\lambda.
\endgathered
\end{equation*}
Therefore, the formula \eqref{HF-3} is proven for $N=0$.

Assume that the formula \eqref{HF-3} is valid for $N=p$.
Then, given that $\mathrm{D}(\{B,A^{p+1}\})\supseteq\mathrm{D}(A)$ and $\{B,A^{p+1}\}\mathrm{D}(A)\subseteq \mathrm{D}(A)$ for $x\in\mathrm{D}(A)$, we get that
\begin{equation*}
\gathered
\{B,A^{p+1}\}R_A(\lambda)x-R_A(\lambda)\{B,A^{p+1}\}x\\=
R_A(\lambda)(\lambda\id-A)\{B,A^{p+1}\}R_A(\lambda)x\\ 
-R_A(\lambda)\{B,A^{p+1}\}(\lambda\id-A)R_A(\lambda)x\\
=R_A(\lambda)\{B,A^{p+2}\}R_A(\lambda)x.
\endgathered
\end{equation*}
Expressing $\{B,A^{p+1}\}R_A(\lambda)x$ from here and substituting into the formula \eqref{HF-3} for $N=p$, we get
\begin{equation}\label{HF-4}
\gathered
BT_tx=T_t\sum\limits_{m=0}^{p}\frac{t^m}{m!}\{B,A^m\}x\\ 
+\frac{1}{2\pi i}\int\limits_{\Gamma}e^{\lambda t}R_A^{p+1}(\lambda)\{B,A^{p+1}\}R_A(\lambda)x\,d\lambda\\
=T_t\sum\limits_{m=0}^{p}\frac{t^m}{m!}\{B,A^m\}x+\frac{1}{2\pi i}\int\limits_{\Gamma}e^{\lambda t}R_A^{p+2}(\lambda)\{B,A^{p+1}\}x\,d\lambda\\
+\frac{1}{2\pi i}\int\limits_{\Gamma}e^{\lambda t}R_A^{p+2}(\lambda)\{B,A^{p+2}\}R_A(\lambda)x\,d\lambda\\
=T_t\sum\limits_{m=0}^{p}\frac{t^m}{m!}\{B,A^m\}x+\frac{1}{2\pi i}\int\limits_{\Gamma}e^{\lambda t}R_A^{p+2}(\lambda)\,d\lambda\{B,A^{p+1}\}x\\
+\frac{1}{2\pi i}\int\limits_{\Gamma}e^{\lambda t}R_A^{p+2}(\lambda)\{B,A^{p+2}\}R_A(\lambda)x\,d\lambda.
\endgathered
\end{equation}
Using the identity (5.22) from \cite{Pazy}:
\begin{equation*}
\gathered
R_{A}^{p+2}(\lambda)=\frac{(-1)^{p+1}}{(p+1)!}\frac{d^{p+1}}{d\lambda^{p+1}}R_A(\lambda)
\endgathered
\end{equation*}
and applying integration by parts $p+1$ times (taking into account that $|e^{\lambda t}|\to 0$ as $\Ree\lambda\to-\infty$ and \eqref{Sector}), we obtain
\begin{equation*}
\gathered
\frac{1}{2\pi i}\int\limits_{\Gamma}e^{\lambda t}R_A^{p+2}(\lambda)\,d\lambda\\ =\frac{1}{2\pi i}\frac{(-1)^{p+1}}{(p+1)!}\int\limits_{\Gamma}
e^{\lambda t}\frac{d^{p+1}}{d\lambda^{p+1}}R_A(\lambda)\,d\lambda\\
=\frac{1}{2\pi i}\frac{(-1)^{p+1}}{(p+1)!}(-1)^{p+1}\int\limits_{\Gamma}\frac{d^{p+1}}{d\lambda^{p+1}}(e^{\lambda t})R_A(\lambda)\,d\lambda\\=
\frac{1}{2\pi i}\frac{t^{p+1}}{(p+1)!}\int\limits_{\Gamma}e^{\lambda t}R_A(\lambda)\,d\lambda=\frac{t^{p+1}}{(p+1)!}T_t.
\endgathered
\end{equation*}
which completes the proof of  \eqref{HF-3}.

We now show that \eqref{HF-3} implies \eqref{HF-2}.
To do this, we estimate the integral in \eqref{HF-3} taking \eqref{Sector} and \eqref{ComIneq} into account:
\begin{equation*}
\gathered
\Big\|\int\limits_{\Gamma_R}e^{\lambda t}R_A^{N+1}(\lambda)\{B,A^{N+1}\}R_A(\lambda)\,d\lambda\Big\|\\
\le \int\limits_{r}^R\frac{K_1K^{N+1}\eta^{N+1}\exp(ts\Ree e^{-i\psi})}{|se^{-i\psi}|^{N+2}}|d(se^{-i\psi})|\\+
\int\limits_{-\psi}^{\psi}\frac{K_1K^{N+1}\eta^{N+1}\exp(tr\Ree e^{i\alpha})}{|re^{i\alpha}|^{N+2}}|d(re^{i\alpha})|\\
+\int\limits_{r}^R\frac{K_1K^{N+1}\eta^{N+1}\exp(ts\Ree e^{i\psi})}{|se^{i\psi}|^{N+2}}|d(se^{i\psi})|\\
\le 2\int\limits_{r}^{\infty}\frac{K_1K^{N+1}\eta^{N+1}\exp(ts\cos(\psi))}{s^{N+1}}\,ds\\+
\int\limits_{-\psi}^{\psi}\frac{K_1K^{N+1}\eta^{N+1}\exp(tr\cos\alpha)}{r^{N+1}}d\alpha\\=
K_1K^{N+1}\eta^{N+1}\Big(2\int\limits_{r}^{\infty}\frac{\exp(ts\cos(\psi))}{s^{N+1}}\,ds\\
+\int\limits_{-\psi}^{\psi}\frac{\exp(tr\cos\alpha)}{r^{N+1}}d\alpha\Big)
\endgathered
\end{equation*}
Applying the change of variables $y=-st\cos(\psi)$, we find
\begin{equation*}
\gathered
\int\limits_{r}^{\infty}\frac{\exp(ts\cos(\psi))}{s^{N+1}}\,ds\\
=(t|\cos(\psi)|)^N\int\limits_{rt|\cos(\psi)|}^{\infty}\frac{\exp(-y)}{y^{N+1}}\,dy\\
\le(t|\cos(\psi)|)^N\frac{e^{-rt|\cos(\psi)|}}{N}(rt|\cos(\psi)|)^{-N}\\ 
=\frac{r^{-N}e^{-rt|\cos(\psi)|}}{N}.
\endgathered
\end{equation*}
Using also the estimate 
\begin{equation*}
\gathered
\int\limits_{-\psi}^{\psi}\frac{\exp(tr\cos\alpha)}{r^{N+1}}d\alpha\le
\frac{2\psi e^{rt}}{r^{N+1}},
\endgathered
\end{equation*}
we get
\begin{equation*}
\gathered
\Big\|\int\limits_{\Gamma_R}e^{\lambda t}R_A^{N+1}(\lambda)\{B,A^{N+1}\}R_A(\lambda)\,d\lambda\Big\|\\
\le
K_1K\eta (K\eta r^{-1})^N\Big(2\frac{e^{-rt|\cos(\psi)|}}{N}+
\frac{2\psi e^{rt}}{r}\Big)
\endgathered
\end{equation*}
Let $r>K\eta$. Then,

\begin{equation*}
\gathered
\Big\|\int\limits_{\Gamma_R}e^{\lambda t}R_A^{N+1}(\lambda)\{B,A^{N+1}\}R_A(\lambda)\,d\lambda\Big\|\to 0
\endgathered
\end{equation*}
as $N\to \infty$ uniformly in $R\ge R_0$, where $R_0$ is a sufficiently large positive number. Therefore, in the formula \eqref{HF-3}, one can pass to the limit $N\to\infty$  and obtain the formula \eqref{HF-2}. The lemma is proven.
\end{proof}

\section{Main result} 
Along with the original impulsive system \eqref{DEq-1}, consider the following impulsive system with constant dwell-time (comparison system)
\begin{equation}\label{CS}
\gathered
\dot z(t)=Az(t),\quad t\ne k\theta,\quad z(0^+)=z_0\in\mathrm{D}(A),\\
z(t^+)=Bz(t)+\sum\limits_{m=1}^{\infty}\frac{(\chi_{\max}+\chi_{k+1})^m}{m!}\{B,A^m\}z(t),\, t=k\theta,
\endgathered
\end{equation}
where $z\in D(A)$ and $\chi_k:=\tau_k-\tau_0-k\theta\le\chi_{\max}$, see \eqref{ADT}.

The main theorem reduces the problem of stability of the initial impulsive system \eqref{DEq-1} to the study of the comparison system \eqref{CS}.

\begin{theorem}
Let the linear operator $A$ be sectorial, the linear operator $B$ be closed, and for all $m\in \Bbb Z_+$ the linear operators $\{B,A^m\}$ be defined as above and such that $\{B,A^m\}\mathrm{D}(A)\subseteq\mathrm{D}(A)$. Assume that inequality \eqref{ComIneq} holds and $\{B,A^m\}T_{\theta-\chi_{\max}}\in L(X)$ for all $m\in\Bbb Z_+$ as well as
\begin{equation}\label{convergence}
\gathered
2e\cdot\chi_{\max} \lim\sup\limits_{m\to\infty}\frac{\|\{B,A^m\}T_{\theta-\chi_{\max}}\|^{1/m}}{m}<1.
\endgathered
\end{equation}

Then, the asymptotic stability of the linear impulsive system \eqref{CS} implies the asymptotic stability of the linear impulsive system \eqref{DEq-1}.
\end{theorem}

\begin{proof}
Let $x_0\in\mathrm{D}(A)$. Then, $x(\tau_1)=T_{\tau_1-\tau_0}x_0$. Therefore, applying the Lemma 1 and the semigroup property, we obtain
\begin{equation*}
\gathered
x(\tau_1^+)=BT_{\tau_1-\tau_0}x_0=BT_{\tau_1-\tau_0-\theta+\chi_{\max}}T_{\theta-\chi_{\max}}x_0\\
=T_{\tau_1-\tau_0-\theta+\chi_{\max}}\sum\limits_{m=0}^{\infty}\frac{(\chi_1+\chi_{\max})^m}{m!}\{B,A^m\}T_{\theta-\chi_{\max}}x_0.
\endgathered
\end{equation*}
Let $$z_0:=\sum\limits_{m=0}^{\infty}\frac{(\chi_1+\chi_{\max})^m}{m!}\{B,A^m\}T_{\theta-\chi_{\max}}x_0.$$ 
Then, using the semigroup property, we have
\begin{equation*}
\gathered
x(\tau_2)=T_{\tau_2-\tau_1}x(\tau_1^+)=T_{\tau_2-\tau_1}T_{\tau_1-\tau_0-\theta+\chi_{\max}}z_0\\=
T_{\chi_2+\chi_{\max}}T_{\theta}z_0.
\endgathered
\end{equation*}
Applying the Lemma 1 again, we get
\begin{equation*}
\gathered
x(\tau_2^+)=BT_{\chi_2+\chi_{\max}}T_{\theta}z_0\\=
T_{\chi_2+\chi_{\max}}\sum\limits_{m=0}^{\infty}\frac{(\chi_2+\chi_{\max})^m}{m!}\{B,A^m\}T_{\theta}z_0.
\endgathered
\end{equation*}
We denote by $\widehat{z}(t)$ the solution to the Cauchy problem for the linear impulsive system \eqref{CS} with the initial condition $\widehat z(0)=z_0$. Then, $x(\tau_2^+)=T_{\chi_2+\chi_{\max}}\widehat{z}(\theta^+)$.
Using the method of mathematical induction, we prove that
\begin{equation}\label{MR}
\gathered
x(\tau_k^+)=T_{\chi_k+\chi_{\max}}\widehat{z}((k-1)\theta^+),\quad k\ge 2.
\endgathered
\end{equation}
For $k=2$, this has already been proven.  Let \eqref{MR} be valid for $k=p$. Then,
\begin{equation*}
\gathered
x(\tau_{p+1}^+)=Bx(\tau_{p+1})=BT_{\tau_{p+1}-\tau_p}x(\tau_p^+)\\
=BT_{\tau_{p+1}-\tau_p}T_{\chi_p+\chi_{\max}}\widehat{z}((p-1)\theta^+)\\=
BT_{\chi_{p+1}+\chi_{\max}}T_{\theta}\widehat{z}((p-1)\theta^+)\\=
T_{\chi_{p+1}{+}\chi_{\max}}\sum\limits_{m=0}^{\infty}\frac{(\chi_{p+1}{+}\chi_{\max})^m}{m!}\{B,A^{m}\}T_{\theta}\widehat{z}((p-1)\theta^+)\\=
T_{\chi_{p+1}+\chi_{\max}}\widehat{z}(p\theta^+).
\endgathered
\end{equation*}
Therefore, \eqref{MR} is valid for $k=p+1$.

From the definition of  $z_0$ by the triangle inequality for norms we estimate
\begin{equation*}
\gathered
\|z_0\|\le \sum\limits_{m=0}^{\infty}\frac{(2\chi_{\max})^m}{m!}\|\{B,A^m\}T_{\theta-\chi_{\max}}\|\|x_0\|=:\mu\|x_0\|.
\endgathered
\end{equation*}
The convergence of the series on the right-hand side follows from the Cauchy criterion, the Stirling formula, and the condition \eqref{convergence}.
For the dwell-time, the estimate $\tau_{k+1}-\tau_k\le\theta+2\chi_{\max}$, $k\in\Bbb Z_+$ follows from \eqref{ADT}.
We denote $$M:=\sup\{\|T_t\|\,\,|\,\,t\in[0,\theta+2\chi_{\max}]\}.$$

Then, \eqref{MR}  implies the estimate
\begin{equation*}
\gathered
\|x(t)\|=\|T_{t-\tau_k}x(\tau_k^+)\|\le\|T_{t-\tau_k}\|\|x(\tau_k^+)\|\\ \le
M\|T_{\chi_k+\chi_{\max}}\widehat{z}((k-1)\theta^+)\|\\\le
M\|T_{\chi_k+\chi_{\max}}\|\cdot\|\widehat{z}((k-1)\theta^+)\|
\endgathered
\end{equation*}
for $t\in(\tau_{k},\tau_{k+1}]$. Since $0\le\chi_k+\chi_{\max}\le 2\chi_{\max}$, we have
$\|T_{\chi_k+\chi_{\max}}\|\le M$, and therefore,
$$\|x(t)\|\le M^2\|\widehat{z}((k-1)\theta^+)\|.$$
Let $\varepsilon>0$. Then, it follows from the stability of the linear impulse system \eqref{CS} that for some $\delta=\delta(\varepsilon)$ for all $k\in\Bbb Z_+$ the inequality $\|z(t)\|<\varepsilon(\mu M^2)^{-1}$ holds. 
Then, $\|x(t)\|<\varepsilon$ is satisfied for all $t\ge 0$ which proves the stability of the original linear impulsive system \eqref{DEq-1}.

The asymptotic stability of the linear impulsive system \eqref{CS} implies that for any $\varepsilon>0$ there exists $k_0(\varepsilon)\in\Bbb Z_+$ such that $$\|\widehat{z}(k\theta^+)\|<\varepsilon  M^{-2},\quad k\ge k_0(\varepsilon).$$ Then, for $t>\tau_{k+1}$, the inequality $\|x(t)\|<\varepsilon$ is satisfied. The theorem is proven.
\end{proof}

\theoremstyle{remark}
\newtheorem*{remark}{Remark}
\begin{remark}
If $T_t$, $t\in\mathbb R$ is a linear group in $L(X)$ or if the sequence of moments of impulsive action $\{\tau_k\}_{k=0}^{\infty}$ which satisfies the following ADT$^+$ condition: 
\begin{equation*}
\gathered
\exists\theta>0,\;\exists\chi_{\max}\ge 0\;\forall k\in\Bbb Z_+
\quad k\theta\le \tau_k-\tau_0\le k\theta+\chi_{\max},
\endgathered
\end{equation*}
the asymptotic stability of the original linear impulsive system \eqref{DEq-1} follows from the stability of the comparison system of the form
\begin{equation}\label{CSbis}
\gathered
\dot z(t)=Az(t),\quad t\ne k\theta,\quad z(0^+)=z_0\in\mathrm{D}(A),\\
z(t^+)=Bz(t)+\sum\limits_{m=1}^{\infty}\frac{\chi_{k+1}^m}{m!}\{B,A^m\}z(t),\quad t=k\theta.
\endgathered
\end{equation}
\end{remark}

\section{Example} 
For $\ell>0$ and $\mu>0$ we consider the linear impulsive system of parabolic equations
\begin{equation}\label{PEq-1}
\gathered
\partial_t x(y,t)=\mu^2\partial^2_{yy}x(y,t)+\bold Ax(y,t),\quad t\ne\tau_k,\\
x(y,t^+)=\bold Bx(y,t),\quad t=\tau_k
\endgathered
\end{equation}
in the state space $X=L^2((0,\ell);\Bbb R^n)$, where 
$x\in C([0,\infty),L^2((0,\ell);\Bbb R^n))\cap
C^1((0,\infty),L^2((0,\ell);\Bbb R^n))$, 
$y\in[0,\ell]$, $t\in\Bbb R_+$,
$\bold A\in\Bbb R^{n\times n}$ and $\bold B\in\Bbb R^{n\times n}$. The initial and boundary conditions are given by
\begin{equation}\label{PEq-2}
\gathered
x(y,0)=x_0(y),\quad x(0,t)=x(\ell,t)=0,
\endgathered
\end{equation}
where $x_0\in H^2(0,\ell)\cap H_0^1(0,\ell)$. Let $\mathrm{D}(A)=H^2(0,\ell)\cap H_0^1(0,\ell)$, $\mathrm{D}(B)=X$
and
\begin{equation*}
\gathered
(Ax)(y):=\mu^2\partial^2_{yy}x(y)+\bold Ax(y),\quad (Bx)(y)=\bold Bx(y).
\endgathered
\end{equation*}
It follows from Theorem 1 that for the asymptotic stability of the linear impulsive system \eqref{PEq-1}-\eqref{PEq-2}, it suffices to check the asymptotic stability of the comparison system
\begin{equation}\label{CS-1}
\gathered
\partial_t z(y,t)=\mu^2\partial^2_{yy}z(y,t)+\bold Az(y,t),\quad t\ne k\theta,\\
z(y,t^+)=\bold Bz(y,t)+\bold G_kz(y,t),\quad t=k\theta,
\endgathered
\end{equation}
where $z\in C^1(\Bbb R_+,X)$, $y\in[0,\ell]$, $t\in\Bbb R_+$,
$\bold G_k\in\Bbb R^{n\times n}$, $\|\bold G_k\|\le\sum_{m=1}^{\infty}\frac{(2\chi_{\max})^m}{m!}\|\{\bold B,\bold A^m\}\|:=\omega$. The initial and boundary conditions are given by
\begin{equation}\label{CS-2}
\gathered
z(y,0)=z_0(y),\quad z(0,t)=z(\ell,t)=0,
\endgathered
\end{equation}
where $z_0\in H^2(0,\ell)\cap H_0^1(0,\ell)$. 
To study the stability of the comparison system \eqref{CS-1}---\eqref{CS-2}, we define a candidate Lyapunov function by
\begin{equation}\label{LF}
\gathered
V(t,z):=\int\limits_0^{\ell}z^{\T}(y)\bold P(t)z(y)\,dy.
\endgathered
\end{equation}
Here, $\bold P\,:\,\,\Bbb R_+\to\Bbb R^{n\times n}$ is a piece-wise continuous and differentiable on the set $\Bbb R_+\setminus\theta\Bbb Z_+$ and $\theta$-periodic map
with values
$P(t)$ in the set of symmetric positive-definite matrices.

The total derivative of this function along the semiflow generated by \eqref{CS-1}---\eqref{CS-2} is
\begin{equation}\label{DLF}
\gathered
\dot V(t,z):=2\mu^2\int\limits_0^{\ell}z^{\T}(y)\bold P(t)\partial_{yy}^2z(y)\,dy\\
+\int\limits_0^{\ell}z^{\T}(y)(\dot{\bold P}(t)+\bold A^{\T}\bold P(t)+\bold P(t)\bold A)z(y)\,dy.
\endgathered
\end{equation}
Applying integration by parts and the Friedrich's inequality, we obtain
\begin{equation*}
\gathered
\int\limits_0^{\ell}z^{\T}(y)\bold P(t)\partial_{yy}^2z(y)\,dy=
-\int\limits_0^{\ell}(\partial_{y}z(y))^{\T}\bold P(t)\partial_{y}z(y)\,dy\\=
-\int\limits_0^{\ell}\|\bold P^{1/2}(t)\partial_{y}z(y)\|^2\,dy
\le-\frac{\pi^2}{\ell^2}\int\limits_0^{\ell}\|\bold P^{1/2}(t)z(y)\|^2\,dy\\=
-\frac{\pi^2}{\ell^2}\int\limits_0^{\ell}z^{\T}(y)\bold P(t)z(y)\,dy
\endgathered
\end{equation*}
Therefore,
\begin{equation*}
\gathered
\dot V(t,z)\le 
\int\limits_0^{\ell}z^{\T}(y)(\dot{\bold P}(t)+(\bold A-\frac{\pi^2\mu^2}{\ell^2}\id)^{\T}\bold P(t)\\ 
+\bold P(t)(\bold A-\frac{\pi^2\mu^2}{\ell^2}\id))z(y)\,dy.
\endgathered
\end{equation*}
We choose $\bold P(t)$ so that for $t\in\Bbb R_+\setminus\theta\Bbb Z_+$, the equality
\begin{equation*}
\gathered
\dot{\bold P}(t)+(\bold A-\frac{\pi^2\mu^2}{\ell^2}\id)^{\T}\bold P(t)+\bold P(t)(\bold A-\frac{\pi^2\mu^2}{\ell^2}\id)=0
\endgathered
\end{equation*}
is satisfied. Then, $$\bold P(t)=e^{\frac{2\pi^2\mu^2(t-k\theta)}{\ell^2}}e^{-\bold A^{\T}(t-k\theta)}\bold P_0e^{-\bold A(t-k\theta)}$$
for $t\in(k\theta,(k+1)\theta]$. 
Assume that there exists a positive-definite matrix $\bold P_0$ that satisfies the matrix inequality
\begin{equation}\label{MIneq}
\gathered
e^{\frac{-2\pi^2\mu^2\theta}{\ell^2}}\Phi^{\T}\bold P_0\Phi+
e^{\frac{-2\pi^2\mu^2\theta}{\ell^2}}(2\omega\|\bold B\bold P_0\|\\
+\omega^2\|\bold P_0\|)e^{\bold A^{\T}\theta}e^{\bold A\theta}\prec\bold P_0,
\endgathered
\end{equation}
where $\Phi=\bold Be^{\bold A\theta}$.
 
At the moments of jumps $t=(k+1)\theta$, $k\in\Bbb Z_+$,
\begin{equation*}
\gathered
V((k+1)\theta^+,z^+)=\int\limits_0^{\ell}(\bold Bz(y))^{\T}\bold P_0(\bold Bz)(y)\,dy\\
+2\int\limits_0^{\ell}(\bold Bz(y))^{\T}\bold P_0(\bold G_kz)(y)\,dy\\
+\int\limits_0^{\ell}(\bold G_kz(y))^{\T}\bold P_0(\bold G_kz)(y)\,dy\\
\le\int\limits_0^{\ell}(z(y))^{\T}\bold B^{\T}\bold P_0\bold Bz(y)\,dy\\
+2\int\limits_0^{\ell}\|\bold B^{\T}\bold P_0\|\omega(z(y))^{\T}z(y)\,dy\\+
\int\limits_0^{\ell}\omega^2\|\bold P_0\|(z(y))^{\T}z(y)\,dy\\=
\int\limits_0^{\ell}(z(y))^{\T}(\bold B^{\T}\bold P_0\bold B
+(2\omega\|\bold B^{\T}\bold P_0\|+\omega^2\|\bold P_0\|)\id)z(y)\,dy\\
<\int\limits_0^{\ell}(z(y))^{\T}\bold P(\theta)z(y)\,dy=V((k+1)\theta,z).
\endgathered
\end{equation*}
Therefore, $V(t,z)$ is a Lyapunov function for the linear impulsive system \eqref{CS-1}-\eqref{CS-2}, hence this system is asymptotically stable. From Theorem 1 we obtain sufficient conditions for the asymptotic stability of the original system \eqref{PEq-1}-\eqref{PEq-2}:
\begin{proposition}
Let the sequence of moments of impulsive action $\{\tau_k\}_{k=0}^{\infty}$ satisfy the ADT condition \eqref{ADT}, $\Phi=\bold Be^{\bold A\theta}$, $\omega=\sum_{m=1}^{\infty}\frac{(2\chi_{\max})^m}{m!}\|\{\bold B,\bold A^m\}\|$, $r_{\sigma}(\Phi)<e^{\pi^2\mu^2\theta/\ell^2}$ and for some positive-definite matrix $\bold P_0$ the inequality \eqref{MIneq} holds. 
Then system \eqref{PEq-1}--\eqref{PEq-2} is asymptotically stable.
\end{proposition}
We consider a numerical example setting $\ell=\pi$, $\mu=1$, $\theta=1$, $\chi_{\max}=0.1$
\begin{equation*}
\gathered
\bold A=
\begin{pmatrix}
1.2&0.1\\
0.1&-3
\end{pmatrix},\quad
\bold B=
\begin{pmatrix}
0.2&0.1\\
-0.1&1.5
\end{pmatrix}.
\endgathered
\end{equation*}
In this case, $\omega\approx 0.1726$ and for the matrix $\bold P_0=\id$, all conditions of the Proposition 1 are satisfied; therefore, the linear impulsive system \eqref{PEq-1} -- \eqref{PEq-2} is asymptotically stable. 
We note that in this case the matrix $\bold A-\frac{\pi^2\mu^2}{\ell^2}\text{id}$ is not a Hurwitz matrix, and the matrix $\bold B$ is not a Schur matrix which means that both continuous and discrete dynamics are unstable. This circumstance as well as non-constant dwell-time is a significant obstacle to the direct application of the Lyapunov function method for the initial system \eqref{PEq-1}---\eqref{PEq-2}.

\section{Conclusion} 
The main theorem allows one to study wide classes of infinite-dimensional systems, for example, systems of parabolic partial differential equations, integro-differential partial differential equations and others. For the comparison system, the problem of construction of a Lyapunov function is much simpler than for the original system since dwell-times are constant. We also note that the obtained stability conditions have a wide range of applicability since they are applicable when the continuous and discrete dynamics are both unstable.
 It is of interest to extend these results to the case when the operator $A$ generates a $C_0$ -- semigroup as well as to relax the assumptions about the operators $A$, $\,B$ and the sequence of commutators $\{B,A^m\}$ that we have applied in Theorem 1.

\addtolength{\textheight}{-12cm}   




\section*{ACKNOWLEDGMENT}

The authors gratefully acknowledge the assistance and valuable suggestions of Dr. Gunther Dirr.


\bibliographystyle{ieeetr}
\bibliography{mybiblio.bib}

\end{document}